\documentclass[12pt]{amsart}
\usepackage{amsmath,amssymb,amsthm,xspace,a4wide,amscd}
\usepackage[all]{xy}
\usepackage{euscript}
\usepackage{xcolor}
\usepackage[plainpages]{hyperref}

\newtheorem{theorem}{Theorem}[section]
\newtheorem{proposition}[theorem]{Proposition}
\newtheorem{lemma}[theorem]{Lemma}
\newtheorem{corollary}[theorem]{Corollary}

\theoremstyle{remark}

\newtheorem{example}[theorem]{Example}

\title{Generalized Imaginary Verma and Wakimoto modules}

\begin{document}
\author[M. G. Alves]{Marcela Guerrini}
\author[I. Kashuba]{Iryna Kashuba}
\author[O. A. H. Morales]{Oscar  Morales}
\author[A. S. Oliveira]{Andr\'e  de Oliveira}
\author[F. J. S. Santos]{Fernando Junior Santos}

\address{ Institute of Mathematics, University of S\~ao Paulo, Caixa Postal 66281 CEP 05314-970, S\~ao Paulo, Brazil} 
\email{marcela.alves@usp.br}
\email{kashuba@ime.usp.br}
\email{oscarhm@ime.usp.br}
\email{andre2.oliveira@usp.br}
\email{fjssmat@ime.usp.br}
\date{}


\maketitle

\begin{abstract}
We develop a
general technique of constructing new irreducible weight modules for any affine Kac-Moody algebra using the parabolic induction, in the case when the Levi factor of a parabolic subalgebra is infinite-dimensional and the central charge is nonzero. Our approach uniforms and generalizes all previously known results with  imposed restrictions on inducing modules.   We also define generalized Imaginary Wakimoto modules which provide an explicit realization for generic generalized Imaginary Verma modules.
\medskip

\medskip


\end{abstract}

\thispagestyle{empty}

\tableofcontents

\section{Introduction}

Representation theory of Kac-Moody algebras has  a vast range of applications in many areas of mathematics and physics. Affine  Lie algebras is the most important family of infinite dimensional Kac-Moody algebras with important connections to quantum field theory and integrable systems.  That is why their representations attract  an increasing attention from both mathematicians and physicists. 

Parabolic induction plays a key role in the representation theory of Lie algebras. In particular, the induction functor is crucial for the study of weight modules for Affine Lie algebras. The induction functor is defined by a choice of a parabolic subalgebra which by definition contains a Borel subalgebra. 

  Borel subalgebras of Affine Lie algebras are essentially classified by the corresponding partitions or quase partitions of the root system  \cite{JK85},  \cite{Fut94, Fut97}, \cite{DFG09}, \cite{FK18}. Verma type modules induced from Borel subalgebras were
studied in \cite{Cox94}, \cite{Cox05}, \cite{FS93}, \cite{Fut92}, \cite{Fut96}, \cite{CF04}, \cite{CF06}, \cite{BBFK} among the others.  Parabolic subalgebras (different from Borel subalgebras) of Affine Kac-Moody algebras are divided into 
two types depending on the dimension of the Levi factor. We are interested in the case when the dimension of the Levi factor is infinite. Such parabolic subalgebras always contain certain infinite dimensional Heisenberg subalgebra.  Modules induced from such
  parabolic subalgebras are called \emph{ generalized Imaginary Verma modules}.
Previously, families of generalized Imaginary Verma modules were studied in  \cite{FK09},  \cite{FK14},   \cite{BBFK},  \cite{FK18}, \cite{FKS19}.
In all these cases  the structure of the  induced modules is well understood  when the central charge (the scalar action of the central element) is nonzero. In particular, it allowed to construct new families of irreducible modules  for Affine Lie algebras using the parabolic induction from certain irreducible modules over the Levi factor of a parabolic subalgebra.  
The largest 
  known family  of the representations of the Levi factor  for which the  corresponding induced module is irreducible was described in \cite{FK18}. In particular, if the Levi factor of a parabolic subalgebra equals the sum of an  infinite dimensional Heisenberg subalgebra and a Cartan subalgebra then the irreducibility is preserved for so-called admissible modules with nonzero central charge. Our first main result extends this statement to arbitrary weight 
  modules over Heisenberg subalgebras. 
 
 For an Affine Lie algebra $\widehat{\mathfrak g}$ with a fixed Cartan subalgebra $H$ consider 
  a parabolic subalgebra $\widehat{\mathfrak{p}}$. Then 
 $\widehat{\mathfrak{p}}$ can be decomposed as follows:
 $$\widehat{\mathfrak{p}}=\widehat{\mathfrak{l}}\oplus \widehat{\mathfrak{u}},$$ 
 where $\widehat{\mathfrak{l}}$ is the Levi factor and $\widehat{\mathfrak{u}}$ is the nilradical.  Assume that
 $\widehat{\mathfrak{l}}=G+H$, where $G$ is the Heisenberg subalgebra of $\widehat{\mathfrak{g}}$, that is the subalgebra generated by all imaginary root subspaces.  The Heisenberg subalgebra $G$ has a natural $\mathbb{Z}$-grading determined by imaginary roots, equivalently by the adjoint action of a certain element $d\in H$ (degree derivation).  Let $c\in H$ be the central element of $\widehat{\mathfrak{g}}$. If $V$ is $\widehat{\mathfrak{g}}$-module (respectively $G$-module) with a scalar action of $c$ then the corresponding scalar is the central charge of $V$.

 For any $\lambda\in H^*$ and a $\mathbb{Z}$-graded  $G$-module $V$ with central charge $\lambda(c)$ we obtain a weight $(G+H)$-module structure on $V$ with the action of $d$ respecting the grading. 
Defining $\widehat{\mathfrak{u}}\cdot V=0$ we get a $\widehat{\mathfrak{p}}$-module structure on $V$. Then 
the parabolic induction defines for each $\lambda\in H^*$ the functor  
  $ \mathbb{I}^{\lambda} $ from the category of $\mathbb{Z}$-graded  $G$-modules with central charge $\lambda(c)$ to the category of weight $\widehat{\mathfrak{g}}$-modules. We have
 
   \begin{theorem} \label{thm-1}  If $\lambda(c) \neq 0$ then  
  the  functor  $ \mathbb{I}^{\lambda} $ preserves the irreducibility. 
\end{theorem}
 
 \medskip
 
 Theorem \ref{thm-1}  serves as a key step to show our main result which gives the irreducibility of generalized Imaginary Verma modules with nonzero central charge.  
Let $\widehat{\mathfrak p}=\widehat{\mathfrak l}\oplus \widehat{\mathfrak u}$ be a parabolic subalgebra of $\widehat{\mathfrak g}$ with an infinite dimensional Levi subalgebra $\widehat{\mathfrak{l}}$.  Then
$\widehat{\mathfrak{l}} = \widehat{\mathfrak{l}^0} + G(\widehat{\mathfrak l})^{\perp}$, where  $\widehat{\mathfrak{l}^0}$ is the Lie subalgebra generated by $H$ and by all real root subspaces of $\widehat{\mathfrak{l}}$ and $G(\widehat{\mathfrak l})^{\perp}$ is the orthogonal complement to the imaginary part of $\widehat{\mathfrak{l}^0}$  with respect to the Killing form.

 If $M$ is a weight (with respect to $H$) $\widehat{\mathfrak{l}^0}$-module with central charge $a$ and $S$ is a $G(\widehat{\mathfrak l})^{\perp}\oplus \mathbb C d$-module with the same central charge and diagonalizable action of $d$, then $M \otimes S$  is an $\widehat{\mathfrak{l}}$-module. Such $\widehat{\mathfrak{l}}$-modules are called \emph{tensor modules}. 
 Let $\mathcal T = \mathcal T_a(\widehat{\mathfrak{l}})$ be the category of tensor $\widehat{\mathfrak{l}}$-modules with central charge $a$. Any  $V\in\mathcal T_a(\widehat{\mathfrak{l}})$
 has a structure of a 
 $\widehat{\mathfrak{p}}$-module with trivial action of the radical $\widehat{\mathfrak{u}}$. Further induction from a  $\widehat{\mathfrak{p}}$-module  to a $\widehat{\mathfrak{g}}$-module defines 
  the  functor $\mathbb{I}_{a, \widehat{\mathfrak{p}}}^{\mathcal T}$ from the category $\mathcal T_a(\widehat{\mathfrak{l}})$ to the category of weight $\widehat{\mathfrak g}$-modules. Our second main result is the following theorem which 
 provides a general tool of constructing new irreducible modules for all Affine Kac-Moody algebras.
  
  \medskip
   
   \begin{theorem} \label{thm-2}  Let $a\in \mathbb{C}\setminus \{0\}$. Then  
  the  functor $\mathbb{I}_{a, \widehat{\mathfrak{p}}}^{\mathcal T}$ preserves the irreducibility.

\end{theorem}
  
  \medskip
  
  Free field realizations of induced modules from a parabolic subalgebra with infinite dimensional  Levi factor $\widehat{\mathfrak l}$ were constructed in \cite{FGM15},  \cite{Mar11}, \cite{Mar13}, \cite{FKS19}. Based on the ideas of \cite{FKS19} and \cite{Futorny-Krizka2021} we define the \emph{generalized Imaginary Wakimoto modules} which realize explicitly generic generalized Imaginary Verma modules. We also define
  the \emph{Imaginary Wakimoto functor} from the category of $\widehat{\mathfrak l}$-modules to the category of $\widehat{\mathfrak g}$-modules which sends an $\widehat{\mathfrak l}$-module to the corresponding generalized Imaginary Wakimoto module.


\section{Quase partitions and parabolic subalgebras}

Let $\widehat{\mathfrak g}$ be an  affine Kac-Moody algebra  
with a Cartan subalgebra $H$ and $1$-dimensional center
$Z=\mathbb{C}c$. We have the  root
decomposition
$$
\widehat{\mathfrak g} = H \oplus (\oplus_{\alpha \in \Delta}
\widehat{\mathfrak g}_\alpha),
$$
where $\Delta$ is  the root
system of $\widehat{\mathfrak g}$ and $\widehat{\mathfrak g}_\alpha = \{ x \in \widehat{\mathfrak g} \, | \, [h, x] = \alpha(h) x
{\text{ for every }} h \in H\}$. 
Let $\pi$ be a fixed basis of the root system $\Delta$.
Then the root system $\Delta$   has a standard partition into positive
and negative roots with respect to $\pi$, $\Delta_+ (\pi)=\Delta_+$ and $\Delta_-(\pi)=\Delta_-$ respectively. 
Let $\delta\in \Delta_+$ be the
indivisible imaginary root.
  Then the set of all imaginary roots is $\Delta^{\rm im}=\{k\delta |
k\in\mathbb{Z}\setminus\{0\}\}$ and the set of all real roots $\Delta^{\rm re}=\Delta\setminus\Delta^{\rm im}$.
Denote by $G$ the Heisenberg subalgebra of $\widehat{\mathfrak g}$ generated by all imaginary root subspaces of $\widehat{\mathfrak g}$. Hence
 $G= G_- \oplus \mathbb{C}c \oplus G_+$, where
$G_{\pm} = \oplus_{k>0} \widehat{\mathfrak{g}}_{\pm k \delta}$. 
 We will denote by $\mathfrak g$ the underlined finite-dimensional simple Lie algebra and by 
 $d$ the element of $H$ such that $\delta(d)=1$ and $\alpha(d)=0$ for any root $\alpha$ of $\mathfrak g$. 

\medskip

\subsection{Quase partitions}

Recall that 
a subset $P\subset \Delta$ is called
 a \emph{partition} of $\Delta$ if $P$ is closed with respect to the sum of roots,  $P\cap (-P)=\emptyset$ and $P\cup
(-P)=\Delta$. 

 Following \cite{FK18} we will say that a subset $P\subset \Delta$ is
 a \emph{quase partition} of $\Delta$ if the following conditions are satisfied:
\begin{itemize}
\item
 $P\cap (-P)=\emptyset$ and $P\cup
(-P)=\Delta$.
\item Let $\widehat{\mathfrak{b}}_P$ be  a Lie  subalgebra of $\widehat{\mathfrak g}$
generated by $H$ and by the root spaces
$\widehat{\mathfrak g}_{\alpha}$ with $\alpha \in P$.  Then for any $\widehat{ \mathfrak{g} }_{\alpha}\subset\widehat{ \mathfrak{b}}_{ P }$ follows $\alpha \in P$.

\end{itemize}

In particular, for any quase partition $P$, the subset of real roots of $P$ is closed with respect to the sum of roots and hence $P\cap \Delta^{\rm re}$ is a partition of real roots. Indeed,  $(P\cap \Delta^{\rm re}) \cap (-P\cap \Delta^{\rm re}) = \emptyset$ and $(P\cap \Delta^{\rm re}) \cup (-P\cap \Delta^{\rm re}) = \Delta^{\rm re}$. Moreover, the second condition  of a quase partition implies that $P\cap \Delta^{\rm re}$ is closed with respect to the addition of roots. On the contrary, the subset of imaginary roots of  a quase partition $P$ is not necessarily a partition of $\Delta^{\rm im}$.  

We will call a subalgebra $\widehat{\mathfrak{b}}_P$ the \emph{Borel subalgebra} of $\widehat{\mathfrak g}$ corresponding to the quase partition $P$ (cf. \cite{DFG09}, \cite{FK18}).
By the results of \cite{JK85} and \cite{Fut94} there are finitely many conjugacy classes of $P\cap \Delta^{\rm re}$ by the Weyl group. On the other hand the number of conjugacy classes of quase partitions is infinite \cite{BBFK}.

\begin{example}

\begin{enumerate}
 \item The partition $P=\Delta_+$ corresponds to the standard Borel subalgebra.\vspace{0.07cm}
 \item  The \emph{natural partition} 
$$\Delta_{+,\textrm{nat}} = \{ \alpha + n\delta \in\Delta \ | \ \alpha \in \dot{\Delta}_{+}, n \in \mathbb{Z}\} \ \mathaccent\cdot\cup \ \{n\delta \ | \ n \in \mathbb{N}\}$$ corresponds to the natural Borel subalgebra 
 $$\widehat{\mathfrak{b}}_{\textrm{nat}} = H \oplus \left(\oplus_{\alpha \in \Delta_{+,\textrm{nat}}} \widehat{\mathfrak{g}}_{\alpha}\right),$$
 where $\dot{\Delta}$ is the root system of $\mathfrak{g}$ and $\dot{\Delta}_+=\dot{\Delta}\cap \Delta_+$.
 \item Each function 
 $\phi: \mathbb{N} \rightarrow \{\pm \}$ defines a  Borel subalgebra $\widehat{\mathfrak{b}}_{\textrm{nat}}^{\phi}$ corresponding to the quase partition
 $$
\{\alpha+k\delta\in\Delta\ |\ \alpha \in \dot{\Delta}_+, k\in \mathbb{Z}\}
\mathaccent\cdot\cup \{n \delta\ |\  n \in \mathbb{N}, \phi(n)=+\}\mathaccent\cdot\cup \{-m \delta\
|\ m \in \mathbb{N}, \phi(m)=-\}
.$$
\end{enumerate}
\end{example}

The subalgebras $\widehat{\mathfrak{b}}_P$ do not exhaust all non-conjugated Borel subalgebras of $\widehat{\mathfrak g}$. To get the full classification one needs to consider all possible triangular decompositions of the Heisenberg subalgebra 
$G$. 
 Using the Killing form, in each subspace $\widehat{\mathfrak{g}}_{k\delta}$ we can choose a commuting basis $x_{k}^1, \ldots, x_{k}^{m_k}$ such that 
  $[x_{k}^i, x_{n}^j]=k\delta_{ij}\delta_{k,-n}c$ for all possible $k, n, i, j$. A triangular decomposition $G=G_-\oplus \mathbb{C} c\oplus G_{+}$, where $G_{\pm}=\sum_{\pm k\in \mathbb{Z}_{>0}}\widehat{\mathfrak{g}}_{k\delta}$, corresponds to a partition of imaginary roots. On the other hand,
  each function 
 $\phi: \mathbb{N} \rightarrow \{\pm \}$ defines a quase partition of imaginary roots and the triangular decomposition $G=G_-^{\phi}\oplus \mathbb{C} c\oplus G_{+}^{\phi}$, where 
 $$G_{\pm}^{\phi}=\sum_{\phi(k)=+}\widehat{\mathfrak{g}}_{\pm k\delta}\oplus \sum_{\phi(k)=-}\widehat{\mathfrak{g}}_{\mp k\delta}.$$
 
 We can obtain different triangular decompositions of $G$
    by dividing  $x_{k}^i$ arbitrarily between $G_+$ and $G_-$ for each integer $k>0$ and  each $i=1, \ldots, m_k$ (if $x_{k}^i\in G_{\pm}$ then $x_{-k}^i\in G_{\mp}$). Such triangular decompositions of $G$ are more general than those corresponding to  quase partitions. They lead to new families of Borel subalgebras of  $\widehat{\mathfrak g}$ (cf. \cite{FK18}).

  For simplicity from now on we will consider the Borel subalgebras $\widehat{\mathfrak{b}}_P$ corresponding to quase partitions $P$ of the root system,  though everything can be easily extended to an arbitrary Borel subalgebra of $\widehat{\mathfrak g}$.

\subsection{Parabolic subalgebras}

 A subalgebra of $\widehat{\mathfrak g}$ is called a {\em parabolic subalgebra} if it contains properly a Borel subalgebra.   There are two types of parabolic subalgebras: those containing the standard Borel subalgebra and those  containing the real part of the natural Borel subalgebra. 
  These parabolic subalgebras are called {\em type I} and {\em type II}  respectively (cf. \cite{Fut97}).  Examples of 
parabolic subalgebras of $\widehat{\mathfrak{g}}$ can be constructed from  parabolic subsets $P \subset \Delta$, which are  closed subsets of $\Delta$ such that $P \cup -P = \Delta$. Given such a parabolic subset $P$, the corresponding parabolic subalgebra  of $\widehat{\mathfrak{g}}$ is generated by $H$ and by all  root subspaces $\widehat{\mathfrak{g}}_{\alpha}$ with $\alpha \in P$.
  A classification of parabolic subsets of $\Delta$  was obtained in \cite{Fut92}, \cite{Fut97}.

 Every parabolic subalgebra $\widehat{\mathfrak{p}}$ of $\widehat{\mathfrak{g}}$ containing a Borel subalgebra $\widehat{\mathfrak{b}}$ has a Levi decomposition $\widehat{\mathfrak{p}} = \widehat{\mathfrak{l}} \oplus \widehat{\mathfrak{u}}$, with the Levi factor $\widehat{\mathfrak{l}}$ and $\widehat{\mathfrak{u}} \subset \widehat{\mathfrak{b}}$. If $\widehat{\mathfrak{p}}$ is of \textit{type II} then $\widehat{\mathfrak{l}}$ is infinite-dimensional.
 In this case the subalgebra $\widehat{\mathfrak{l}}$ contains a sum of certain affine Lie subalgebras and a subalgebra of the Heisenberg algebra $G$. Further suppose $\widehat{\mathfrak p}$ is such parabolic subalgebra of $\widehat{\mathfrak g}$ which is not a Borel subalgebra. Then its real part can be described as follows  \cite{Fut97}, \cite{FK09}.  Assume for simplicity  that a parabolic subalgebra $\widehat{\mathfrak p}$ of $\widehat{\mathfrak g}$ contains the whole natural Borel subalgebra (not just its real part). Then such parabolic subalgebra $\widehat{\mathfrak p}$ has
 the Levi decomposition 
  $\widehat{\mathfrak p} =\widehat{\mathfrak l}\oplus  \widehat{\mathfrak u}$,
  where $\widehat{\mathfrak l}$ is the infinite dimensional Levi factor.  Denote by $\widehat{\mathfrak l}^0$ the Lie subalgebra of $\widehat{\mathfrak l}$ generated by all its real root subspaces and by $H$.
   Set $G({\widehat{\mathfrak l}})$ to be a subalgebra of $\widehat{\mathfrak l}^0$ spanned by its imaginary root subspaces. Let $G(\widehat{\mathfrak l})^{\perp}\subset G$
 be the orthogonal complement of  $G(\widehat{\mathfrak l})$ in $G$ with respect to the Killing form, that is 
$G=G(\widehat{\mathfrak l})+G(\widehat{\mathfrak l})^{\perp}$, $[G(\widehat{\mathfrak l})^{\perp}, \widehat{\mathfrak l}^0]=0$ and  $\widehat{\mathfrak l}^0\cap G(\widehat{\mathfrak l})^{\perp}=\mathbb{C} c$.   
The Levi subalgebra $\widehat{\mathfrak l}$  may contain the whole $G(\widehat{\mathfrak l})^{\perp}$ or a part of it. In the former case  $\widehat{\mathfrak l}$ contains $G$ and $\widehat{\mathfrak l} =\widehat{\mathfrak l}^0+ G(\widehat{\mathfrak l})^{\perp}$. 

  Let $n$ be the rank of the underlined simple finite dimensional Lie algebra $\mathfrak g$ and\\
   $\{\alpha_1, \ldots, \alpha_n\}\subset \pi$  the set of simple roots of  
$\mathfrak g$.  Set   $I = \{1, \dots, n\}$ and consider a subset $\omega\subset I$.  
Denote by $\dot \Delta^{\omega}$  the root system generated by the 
roots  $\alpha_i, i\in \omega$ and set
$$
\Delta^{\omega} = \{ \alpha + n\delta \in \Delta\ |\ \alpha \in \dot \Delta^\omega, n \in \mathbb{Z}\} \cup
\{n\delta \ |\ n \in \mathbb{Z} \setminus \{0\} \},
$$
$
\Delta^{\omega, \rm re} = \Delta^{\omega} \cap \Delta^{\rm re}.
$
Then the real roots of $\widehat{\mathfrak l}$ are conjugated to $\Delta^{\omega, \rm re}$ for some choice of $\omega$ (cf. \cite{Fut94}).

\medskip

\section{Generalized Imaginary Verma modules}

\subsection{Imaginary induction functor}
Let $a \in \mathbb{C}$ and $V$ be a $\widehat{\mathfrak{p}}$-module with central charge $a$ and with $\widehat{\mathfrak{u}} \cdot V = 0$. Define the \emph{generalized Verma type module} $$M_{a,\widehat{\mathfrak p}}(V)= \textrm{Ind}_{\widehat{\mathfrak{p}}}^{\widehat{\mathfrak{g}}}(V)=U(\widehat{\mathfrak g})\otimes_{U(\widehat{\mathfrak p})}V.$$

 When $\widehat{\mathfrak p}=\widehat{\mathfrak b}$ is a Borel subalgebra and $V\simeq \mathbb{C}$ is a one-dimensional $\widehat{\mathfrak b}$-module such that $hv=\lambda(h)v$ for all $v\in V$ and $h\in H$, 
 we get a \emph{Verma type module} $M_{\widehat{\mathfrak b}}(\lambda)=U(\widehat{\mathfrak g})\otimes_{U(\widehat{\mathfrak b})}\mathbb{C}$ 
(\cite{Fut97}).
 
To emphasize  that we will consider the case of  a parabolic subalgebra $\widehat{\mathfrak p}$ of 
type II we call $M_{a,\widehat{\mathfrak{p}}}(V)$ \emph{generalized Imaginary Verma module} with central charge $a$. 
 We have the following well known particular cases of the construction above. 
  
  \begin{enumerate}
 \item     In the particular case when $\widehat{\mathfrak{l}}= G+H$ 
     and $V$ is a  $(G+H)$-module, the module $M_{a,\widehat{\mathfrak{p}}}(V)$ is a
      \emph{ generalized loop module} with central charge $a$.     
     \item Taking $V$ to be a Verma module for $G+H$ with highest weight $\lambda\in H^*$ with respect to the triangular decomposition $G=G_-\oplus \mathbb{C} c\oplus G_+$, the corresponding module 
      $M_{a,\widehat{\mathfrak{p}}}(V)$ is an \emph{imaginary Verma module} \cite{Fut94}.
        \item Finally, taking $V$ to be a Verma module for $G+H$ with highest weight $\lambda\in H^*$ with respect to the triangular decomposition $G=G_-^{\phi}\oplus \mathbb{C} c\oplus G_+^{\phi}$, the corresponding module 
      $M_{a,\widehat{\mathfrak{p}}}(V)$ is a \emph{ $\phi$-Imaginary Verma module} \cite{BBFK}.
    \end{enumerate}

 For any Lie algebra ${\mathfrak{a}}$ with a Cartan subalgebra $\mathfrak{h}$ denote by 
 $\mathcal{M}({\mathfrak{a}})$ the category of all weight (i.e. $\mathfrak{h}$-diagonalizable) ${\mathfrak{a}}$-modules,
 and by $\mathcal{M}(a,{\mathfrak{a}})$ its subcategory of all weight ${\mathfrak{a}}$-modules with central charge $a$. Then
  we have the  following \emph{Imaginary (induction) functor}
  
   \begin{align*}
  \mathbb{I}_{a, \widehat{\mathfrak{p}}} \colon \mathcal{M}(a, \widehat{\mathfrak{l}}) \rightarrow \mathcal{M}(a, \widehat{\mathfrak{g}}),
   \end{align*}
   which sends an $\widehat{\mathfrak{l}}$-module $V$ to the $\widehat{\mathfrak{g}}$-module   $M_{a,\widehat{\mathfrak{p}}}(V)$.

For any positive integer $k$, denote $G_k=\widehat{\mathfrak{g}}_{k\delta}\oplus \mathbb{C} c\oplus \widehat{\mathfrak{g}}_{-k\delta}$ . We say that a $G_k$-module $M$ is 
$U(\widehat{\mathfrak{g}}_{k\delta})$-surjective (respectively $U(\widehat{\mathfrak{g}}_{-k\delta})$-surjective) if for any two elements $v_1, v_2\in M$ there exist $v \in M$ and  $u_1, u_2\in U(\widehat{\mathfrak{g}}_{k\delta})$ (respectively, $u_1, u_2\in U(\widehat{\mathfrak{g}}_{-k\delta})$) such that  $v_i=u_i v$, $i=1,2$. 
A 
$G(\widehat{\mathfrak l})^{\perp}$-module $M$ is admissible if for any positive integer $k$,  any its cyclic $G_k\cap G(\widehat{\mathfrak l})^{\perp}$-submodule $M'\subset M$ is $U(\widehat{\mathfrak{g}}_{k\delta}\cap G(\widehat{\mathfrak l})^{\perp})$-surjective or $U(\widehat{\mathfrak{g}}_{-k\delta}\cap G(\widehat{\mathfrak l})^{\perp})$-surjective. Recall the following   irreducibility criterion for admissible tensor modules which follows from \cite[Theorem 1, Corollary 3]{FK18}.

\begin{theorem} \cite{FK18}\label{thm-FK}
Let $a\in \mathbb{C}$,
$\widehat{\mathfrak{l}} = \widehat{\mathfrak{l}^0} + G(\widehat{\mathfrak l})^{\perp}$ and $V \simeq M \otimes S$, 
 where $M$ is an irreducible weight $\widehat{\mathfrak{l}^0}$-module and $S$ is an admissible irreducible $\mathbb Z$-graded $G(\widehat{\mathfrak l})^{\perp}$-module. Then  $\mathbb{I}_{a, \widehat{\mathfrak{p}}}(V) = \mathbb{I}_{a, \widehat{\mathfrak{p}}}^{\mathcal{T}}(V)$ is irreducible $\widehat{\mathfrak{g}}$-module if $a \neq 0$.
\end{theorem}

\begin{example}
Let $a\in \mathbb{C}$, $\widehat{\mathfrak{p}}$ is a parabolic subalgebra of $\widehat{\mathfrak{g}}$ with the Levi factor
$\widehat{\mathfrak{l}}=G+H$. Let $L(a)$ be the highest weight $G$-module with respect to the decomposition $G=G_-\oplus \mathbb C c\oplus G_+$ and with central charge $a$ (irreducible if $a\neq 0$). Define an $\widehat{\mathfrak{l}}$-module structure on $L(a)$ by fixing the scalar action of $H\cap \mathfrak{g}$ and defining the weight structure on $L(a)$ with respect to the derivation $d$.  Then $\mathbb{I}_{a, \widehat{\mathfrak{p}}}(L(a))$ is irreducible $\widehat{\mathfrak{g}}$-module if $a \neq 0$. It is isomorphic to an imaginary Verma module \cite{Fut97}. 

 Let $L(a)^-$ be the lowest weight $G$-module  with respect to the same triangular decomposition and with central charge $a$ (irreducible if $a\neq 0$), and $V=L(a)\otimes L(a)^-$. Then $V$ is neither diagonal (in the sense of \cite{BBFK})  nor admissible. 
\end{example}

We will show that the condition of admissibility can be lifted in Theorem \ref{thm-FK}. 
In the following subsection we will prove it for an arbitrary generalized loop module.

\medskip

\subsection{Irreducibility of generalized loop modules}

Let $I=\{1, \ldots, n\}$ and $\alpha_1, \ldots, \alpha_n$ be simple roots of $\mathfrak{g}$.
For any
subset $\omega\subset I$, let $Q^{\omega}_{\pm}$ denote a
semigroup of $H^*$ generated by $\pm \alpha_i$, $i\in \omega$. Set   
$Q^{\omega} = \oplus_{j \in \omega} \mathbb{Z} \alpha_j
\oplus \mathbb{Z}\delta$  and if $\omega = I$, denote $Q^{I}$  by $Q$. We can see that $U(\widehat{\mathfrak{g}})$ is a $Q$-graded algebra, then $$U(\widehat{\mathfrak{g}}) = \bigoplus\limits_{\phi \in Q} U(\widehat{\mathfrak{g}})_{\phi},$$ and the elements of $U(\widehat{\mathfrak{g}})_{\phi}$ are called homogeneous elements of $U(\widehat{\mathfrak{g}})$ of homogeneous degree $\phi \in Q$.

 Let $\alpha\in Q^{\omega}_{-}$
and $\alpha = -\sum_{j\in \omega} k_j \alpha_j$,
where each $k_j$ is in $\mathbb{Z}_{\geq 0}$. We define the $\omega$-height of $\alpha$ to be 
$\mathsf{ht}_{\omega}(\alpha) = \sum_{j=1}^n k_j$.  

For a parabolic subalgebra $\widehat{\mathfrak{p}}= \widehat{\mathfrak{l}} \oplus \widehat{\mathfrak{u}}$, the subalgebra
$\widehat{\overline{\mathfrak{u}}}$ satisfying 
$\widehat{\mathfrak{g}}= \widehat{\mathfrak{p}} \oplus \widehat{\overline{\mathfrak{u}}}$ is called {\it the opposite radical}
of $\widehat{\mathfrak{p}}$. Observe that for any $\widehat{\mathfrak{p}}$-module $V$ the induced module $M_{a, \widehat{\mathfrak{p}}}(V)\simeq U(\widehat{\overline{\mathfrak{u}}})\otimes_{\mathbb{C}}V$ as a vector space.

 Assume $\omega=\omega_{\widehat{\mathfrak{u}}}$ consists of all indices $i$ for which $\alpha_i$ is a simple root of $\,\widehat{\mathfrak{g}}\,$ with $\,\widehat{\mathfrak{g}}_{\alpha}\in\widehat{\overline{\mathfrak{u}}}.$ Let  $v\in M_{a, \widehat{\mathfrak{p}}}(V)$ be a
nonzero weight element ($v \in M_{a, \widehat{\mathfrak{p}}}(V)_{\mu}$, for $\mu \in H^{\ast}$).  
Then $$v=\sum_{i} u_iv_i,$$ for some finite set of indices,  
where  for all $i$, $u_i\in U(\widehat{\overline{\mathfrak{u}}})$ are linearly independent  elements and 
$v_i\in V$.  Since $v$ is a weight vector then each $u_i$ is a homogeneous element of $U(\widehat{\overline{\mathfrak{u}}})$ of  homogeneous degree  $\phi_i\in Q$.  Each $\phi_i$ can be written as 
$\phi_i=\psi_i+\tau_i$, where  $\psi_i\in Q_-^{\omega_{\widehat{\mathfrak{u}}}}$ and $\tau_i\in Q^{I\setminus \omega_{\widehat{\mathfrak{u}}}}$.
 Moreover, due to homogeneity  we have $\psi_i=\psi_j$
for all $i,j$. 
Denote by $\psi$ this common restriction to  $Q_-^{\omega_{\widehat{\mathfrak{u}}}}$.

  The {\em $\widehat{\mathfrak{u}}$-height} of $v$  will be defined as the $\omega$-height of $\psi$   and will be denoted by 
  $\mathsf {ht}_{\widehat{\mathfrak{u}}}(v)$.

We assume now that $\omega=\emptyset$ and 
 $\widehat{\mathfrak{l}}=G+H$.  
 We will say that
 $V$ is a weight $(G+H)$-module if $V$ is a $\mathbb{Z}$-graded $G$-module with diagonalizable action of 
$H$, where the action of $d$ is compatible with the $\mathbb{Z}$-grading. An irreducible weight $(G+H)$-module is determined by an irreducible $\mathbb{Z}$-graded $G$-module $V$ and a weight $\lambda\in H^*$ with the action of
$h\in \mathfrak h=H\cap \mathfrak g$ and $c$ on $V$ by  $\lambda(h)$ and  $\lambda(c)$ respectively, while the action of $d$ is given  by $\lambda(d)$ on a chosen graded component of $V$.

If  $\lambda(c)=a$, 
we will denote the corresponding generalized Imaginary Verma module
 $M_{a, \widehat{\mathfrak{p}}}(V)$ by $M_{a}(V)$. For each $\lambda\in H^*$ we have a functor 
  $ \mathbb{I}^{\lambda} $ from the category of $\mathbb{Z}$-graded  $G$-modules with central charge $\lambda(c)$ to the category of $\widehat{\mathfrak{g}}$-modules, which sends a $G$-module $V$ to  $M_{\lambda(c)}(V)$ 
  with a natural weight structure with respect to $H$ as described above. 
 Our goal is to show that Theorem \ref{thm-FK} holds for modules  $M_{a}(V)$ in full generality without the assumption of admissibility. Namely, we have

  \begin{theorem} \label{thm-loop}
Let  $V$ be an irreducible weight $G+ H$-module and $a \neq 0$. Then the $\widehat{\mathfrak{g}}$-module $M_{a}(V)$ is irreducible. In particular, for any $\lambda\in H^*$ the functor  $ \mathbb{I}^{\lambda} $ preserves irreducibility if $\lambda(c)\neq 0$.
    \end{theorem}
 
 The following is the key lemma which generalizes \cite[Lemma 2]{FK18}.

 \begin{lemma}\label{lem-heis}
 Let  $V$ be a weight $(G+H)$-module with nonzero scalar action of $c$: $cv=av$ for some $a\in \mathbb{C}\setminus \{0\}$ and any $v\in V$. Also, let $w_1, \ldots, w_m\in V$ be  nonzero elements such that $w_i\in V_{\mu_i}$ for $\mu_i\in H^*$ and ${\mu}_i-{\mu}_{i+1}\in \mathbb{Z}_{>0}\delta$ for all $i$. Let $k_{i}$ be such that $\mu_1-\mu_i = k_{i}\delta$, $i=1, \ldots, m$.
 Then for a sufficiently large $N$ and for any choice of nonzero elements 
 $x_k\in \widehat{\mathfrak{g}}_{k\delta}$ such that $[x_k, x_{-k}]=kc\neq 0$, at least one of the sums
 $$ \sum_{i=1}^m x_{N-k_i} w_i, \,\,\,\,   \sum_{i=1}^m x_{-N-k_i} w_i$$
 is nonzero.
\end{lemma}

\begin{proof}
Suppose that 
 $$ \sum_{i=1}^m x_{N-k_i} w_i =\sum_{i=1}^m x_{-N-k_i} w_i=0$$
 for all  sufficiently large $N$.  We proceed by the induction on $m$. First assume that $m=1$, that is $x_N w=x_{-N}w=0$ for all  sufficiently large $N$ and some nonzero $w\in V$. Then we immediately get a contradiction as for any such $N$ we have
 $$0=[x_N, x_{-N}]w=Ncw=Naw\neq 0.$$
 To explain the idea of the induction step consider the case $m=2$ and assume that 
 $$ x_{N} w_1+ x_{N-k_2}w_2 = x_{-N} w_1+ x_{-N-k_2}w_2 =0$$
  for all  sufficiently large $N$. Introduce the following operators for all pairs $(r,s)\in \mathbb{Z}^2$:
  $$\Omega_{N-r,N-s}:= x_{N-r} x_{-N-s}- x_{N-s} x_{-N-r}.$$  Note that if $N-r \neq N+s$, then $\Omega_{N-r,N-s} = x_{-N-s} x_{N-r}- x_{N-s} x_{-N-r}.$ 
  Applying $x_{-N}$ and $x_N$ respectively to the equalities above we get
  $$x_{-N}x_{N} w_1+ x_{-N}x_{N-k_2}w_2 =x_{N} x_{-N} w_1+ x_{N}x_{-N-k_2}w_2 =0.$$ 
  Subtracting these equalities we obtain
   $$(x_{-N}x_{N}-x_{N} x_{-N}) w_1+ (x_{-N}x_{N-k_2}-x_{N}x_{-N-k_2})w_2 =-Naw_1+\Omega_{N-k_2, N}(w_2) =0,$$  
  implying $\Omega_{N-k_2, N}(w_2) =Na w_1$. Let us show that $\Omega_{N-k_2, N}(w_1)=0$. Indeed,
  $$\Omega_{N-k_2, N}(w_1)=(x_{-N}x_{N-k_2}-x_{N}x_{-N-k_2})w_1=x_{N-k_2}(x_{-N}w_1)-x_{-N-k_2}(x_{N}w_1)=$$
  $$=x_{N-k_2}(-x_{-N-k_2}w_2)-x_{-N-k_2}(-x_{N-k_2}w_2)=0,$$
  since $k_2>0$. Choose $N_1>>N_2>>0$. Then we have
  $$0=\Omega_{N_1-k_2, N_1}( x_{N_2} w_1+ x_{N_2-k_2}w_2)=x_{N_2}\Omega_{N_1-k_2, N_1}(w_1)+x_{N_2-k_2}\Omega_{N_1-k_2, N_1}(w_2)=x_{N_2-k_2}(N_{1}a w_1)$$
  and $x_{N_2-k_2}w_1=0$ for any sufficiently large $N_2$. Similarly,
  $$0=\Omega_{N_1-k_2, N_1}( x_{-N_2} w_1+ x_{-N_2-k_2}w_2)=x_{-N_2}\Omega_{N_1-k_2, N_1}(w_1)+x_{-N_2-k_2}\Omega_{N_1-k_2, N_1}(w_2)=x_{-N_2-k_2}(N_{1}a w_1)$$
 and $x_{-N_2-k_2}w_1=0$ for any sufficiently large $N_2$. We conclude that $x_N w_1=x_{-N}w_1=0$  
  for all sufficiently large $N$, which a contradiction as $a\neq 0$. 
  
  Suppose that the statement holds for any set of $<m$ elements.  Applying $x_{-N}$ and $x_N$ respectively to the equalities above we get
  $$x_{-N} \sum_{i=1}^m x_{N-k_i} w_i =x_N \sum_{i=1}^m x_{-N-k_i} w_i=0.$$ 
  Subtracting these equalities we obtain
 $$(x_{-N}x_{N}-x_{N} x_{-N}) w_1+ \sum_{i=2}^m \Omega_{N-k_i, N}w_i =-Naw_1+ \sum_{i=2}^m \Omega_{N-k_i, N}(w_i) =0$$
 and $$\sum_{i=2}^m \Omega_{N-k_i, N}(w_i) =Naw_1.$$
  We also have for each $i=2, \ldots, m$
  $$\Omega_{N-k_i, N}(w_1)=(x_{-N}x_{N-k_i}-x_{N}x_{-N-k_i})w_1=x_{N-k_i}(x_{-N}w_1)-x_{-N-k_i}(x_{N}w_1)=
$$
  $$=x_{N-k_i}(-\sum_{j=2}^m x_{-N-k_j}w_j)-x_{-N-k_i}(-\sum_{j=2}^m x_{N-k_j}w_j)=\sum_{j=2, j\neq i}^m (-x_{N-k_i}x_{-N-k_j}+x_{-N-k_i}x_{N-k_j})w_j$$
  
  and 
    \begin{equation}\label{system}
  \Omega_{N-k_i, N}(w_1)=\sum_{j=2, j\neq i}^m \Omega_{N-k_j, N-k_i}(w_j), \,\,\, i=2,\ldots, m.
     \end{equation} 
  Note that
  $$\Omega_{N-k_j, N-k_i}+\Omega_{N-k_i, N-k_j}=0$$
  for all $i,j$ and $[\Omega_{N-k_j, N-k_i}, \Omega_{N-k_s, N-k_r}]=0$ for all $i,j,r,s=2, \ldots, m$.
We can view the equalities \eqref{system} as  a linear system of $m-1$ equations with operator commuting coefficients $\Omega_{N-k_j, N-k_i}$. Note that $m-2$ equations contain $w_m$.

Choose arbitrary sufficiently large  $t_1>>t_2>>\ldots >> t_{m-2}>>N$ and multiply the $j$-th equation by  $x_{t_j-k_m}$ for each $j=1, \ldots,  m-2$. Recall that we have 
$$ \sum_{i=1}^m x_{t_j-k_i} w_i=\sum_{i=1}^m x_{-t_j-k_i} w_i=0, \,\,\, j=1, \ldots,  m-2,$$
by our assumption. Then we can  replace  $x_{t_j-k_m}w_m$  in the $j$-th equation by $- \sum_{i=1}^{m-1} x_{t_j-k_i} w_i$.  We obtain $m-1$ equations with parameters $t_j$, $j=1, \ldots,  m-2$. 
 Using multiplication and addition with will eliminate $w_2, \ldots, w_{m-1}$ and obtain the equality
 $\Omega(w_1)=0,$ 
  where $\Omega$ is a certain polynomial operator spanned by the products of $x_{t_j- l_j}$ for some $l_j\geq 0$, 
  $j=1, \ldots,  m-2$, and $\Omega_{N-k_r, N-k_s}$, $r,s=1, \ldots, m$.  Let $K>> t_1$. Then 
  $$0=\Omega \sum_{i=1}^m x_{\pm K-k_i} w_i= \sum_{i=1}^m x_{\pm K-k_i}  \Omega(w_i)=\sum_{i=2}^m x_{\pm K-k_i}  \Omega(w_i).$$
  If $\Omega(w_i)\neq 0$  for at least one index $i=2, \ldots, m$, then we can apply the induction hypothesis and obtain the contradiction. It remains to consider the case when $\Omega(w_i)= 0$ for all $i$. 
  
   Choose the largest index $n$ for which $x_n$ appears in $\Omega$ and consider the equality
  $$ \sum_{i=1}^m x_{-t_1-k_i} w_i=0.$$
  Applying $\Omega$ we obtain
  $$0=[\Omega, \sum_{i=1}^m x_{-t_1-k_i}] w_i=[\Omega, x_{-t_1}]w_1,$$
  implying that $\Omega^1 w_1=0$ where $\Omega^1$ is a polynomial operator with degree of $x_{t_1}$ reduced by $1$ (as $[x_{t_1}, x_{-t_1}]=at_1$). Again we can assume that $\Omega^1 w_i=0$ for all $i$ (otherwise apply the induction hypothesis). 
  Continuing  this way we eliminate all $x_{t_j}$, $j=1, \ldots,  m-2$ and obtain a polynomial operator  $\Omega^2$
  which is spanned by the products of $\Omega_{N-k_r, N-k_s}$, $r,s=1, \ldots, m$, and $\Omega^2w_1=0$. Again we can assume that $\Omega^2 w_i=0$ for all $i$. Note that each monomial in $\Omega^2$ contains exactly one factor of type $\Omega_{N-k_i, N}$, $i>2$ and these factors do not comute. All other factors that appear in $\Omega^2$ commute. Hence, each monomial in $\Omega^2$ contains at most one factor $x_N$ and  
    $\Omega^2=f x_N+g$, where $f$ and $g$ do not contain $x_N$ and contain only factors $x_k$  with $k<N$. 
   Consider the equality
  $$ \sum_{i=1}^m x_{-N-k_i} w_i=0$$
  and apply $\Omega$. Then we get 
   $$0=[\Omega^2, \sum_{i=1}^m x_{-N-k_i}] w_i=[\Omega^2, x_{-N}]w_1=f[x_N, x_{-N}]w_1=Na f w_1.$$   
We conclude that $f w_1=0$ and hence $f w_i=0$ for all $i$. Choose the factor in $f$ with the largest index and repeat the argument. Continuing we will obtain $w_1=0$ which is a contradiction. This completes the proof. 
\end{proof}

\medskip

Now we proceed with the proof of Theorem \ref{thm-loop}.
Let  $V$ be an irreducible weight $(G+ H)$-module. Consider the $\widehat{\mathfrak{g}}$-module $M_{a}(V)$ and assume $a \neq 0$. 
Let  $v\in M_{a}(V)$ be a
nonzero  element. We will show that $v$ generates  $M_{a}(V)$. We can assume that $v$ is a 
weight element
and $$v=\sum_{i} u_iw_i,$$
where $u_i\in U(\widehat{\overline{\mathfrak{u}}})$ are linearly independent homogeneous elements and 
$w_i\in V_{\mu_{i}}$  for all $i$.   Let $R=\mathsf {ht}_{\widehat{\mathfrak{u}}}(v)$. We will proceed by the induction on $R$.  Suppose first that $R=1$.   Then $u_i=x_{-\beta+r_i\delta}\in \widehat{\mathfrak{g}}_{-\beta+r_i\delta}$ for some simple root $\beta$ of $\widehat{\mathfrak{g}}$. 
For an integer $N$ consider a nonzero
 $x_{\beta+ N\delta}\in \widehat{\mathfrak{g}}_{\beta+ N\delta}$. Then 
 $$x_{\beta+ N\delta} v= x_{\beta+ N\delta}\sum_{i} u_i w_i=\sum_{i}[x_{\beta+ N\delta}, u_i] w_i=\sum_i x_{(N+r_i)\delta}w_i \in V.$$
By Lemma \ref{lem-heis}, $x_{\beta+ N\delta} v\neq 0$ if $|N|$ is sufficiently large.
 Since $V$ is irreducible then $v$ generates $M_{a}(V)$. 

Assume now that $R>1$ and assume that any nonzero weight element of height $<R$ generates $M_{a}(V)$. We follow the proof of  \cite[Lemma 5.3]{BBFK}.   Suppose that  $u_i=x_{-\beta_{i1}+ r_{i1}\delta}^{p_{i1}}\ldots x_{-\beta_{is_i}+ r_{is_i}\delta}^{p_{is_{i}}}$ for some simple roots 
$\beta_{i1}, \ldots, \beta_{is_i}$ and some integers $r_{i1}, \ldots, r_{is_i}$. We can  assume without loss of generality that the summands are  indexed in such a way that $r_{1s_1} \leq r_{2s_2} \leq \ldots $  and if $r_{ij} = r_{1s_{1}}$ then $p_{1s_{1}}$ is minimal. Applying $x_{\beta_{11}+ N\delta}$   with sufficiently large $|N|$ we get 
$x_{\beta_{11}+ N\delta} v\neq 0$ again by Lemma \ref{lem-heis}. Since
$\mathsf {ht}_{\widehat{\mathfrak{u}}}(x_{\beta_{11}+ N\delta}v)=R-1$, the element $x_{\beta_{11}+ N\delta}v$ generates $M_{a}(V)$ by induction. If  $u_i$ contains a factor   $x_{-\beta+ r\delta}$ with non-simple 
 positive root $\beta$ for some $i$, then choose a simple root $\alpha$  such that $\alpha- \beta$ is a root 
 and apply $x_{\alpha+ N\delta}$  with sufficiently large $|N|$. Then  $x_{\alpha+ N\delta}v\neq 0$ and $\mathsf {ht}_{\widehat{\mathfrak{u}}}(x_{\alpha+ N\delta}v)=R-1$. Again the induction completes the proof.

\begin{corollary}\label{cor-part}
Let $\widehat{\mathfrak p} =\widehat{\mathfrak l}\oplus  \widehat{\mathfrak u}$ be a parabolic subalgebra of 
$\widehat{\mathfrak g}$ of type II and  $\widehat{\mathfrak l}=G' + H$, where $G'$ is a subalgebra of $G$.  
Let $a\in \mathbb{C}\setminus \{0\}$ and let
$V$ be an irreducible weight $\widehat{\mathfrak l}$-module of central charge $a$. Then 
 $M_{a, \widehat{\mathfrak{p}}}(V)$ is irreducible, and hence the restriction of the functor $\mathbb{I}_{a, \widehat{\mathfrak{p}}}$ to the category of weight $\widehat{\mathfrak l}$-modules with  central charge $a$ preserves the irreducibility. 
\end{corollary}

\begin{proof}
Follows from Theorem \ref{thm-loop} and from the fact that a highest weight module of an infinite dimensional Heisenberg 
subalgebra is irreducible if the central charge is nonzero.

\end{proof}

Theorem \ref{thm-loop} and  is a key step to approach the irreducibility of generalized Imaginary Verma modules, which  
 we are going to do in  the following subsection.

\medskip

\subsection{Irreducibility of generalized Imaginary Verma modules}

Recall that $\widehat{\mathfrak p}=\widehat{\mathfrak{l}} \oplus \widehat{\mathfrak{u}} $ is  a parabolic subalgebra of $\widehat{\mathfrak g}$ of type II, $\widehat{\mathfrak{l}} = \widehat{\mathfrak{l}^0} + G(\widehat{\mathfrak l})^{\perp}$, where  $\widehat{\mathfrak{l}^0}$ is the Lie subalgebra generated by $H$ and by all real root subspaces of $\widehat{\mathfrak{l}}$ and $G(\widehat{\mathfrak l})^{\perp}$ is the orthogonal complement to the imaginary part of $\widehat{\mathfrak{l}^0}$  with respect to the Killing form.

Tensor $\widehat{\mathfrak{l}}$-modules have the form  $M \otimes S$, where  
  $M$ is a weight (with respect to $H$) $\widehat{\mathfrak{l}^0}$-module with central charge $a$ and $S$ is a $G(\widehat{\mathfrak l})^{\perp}\oplus \mathbb C d$-module with the same central charge and diagonalizable action of $d$.  Modules over $G(\widehat{\mathfrak l})^{\perp}\oplus \mathbb C d$ correspond to $\mathbb Z$-graded modules  (with grading compatible with the $\mathbb Z$-grading of $ G(\widehat{\mathfrak l})^{\perp}$)     with a fixed scalar action of $d$ on a chosen nonzero component. We will say that a $G(\widehat{\mathfrak l})^{\perp}\oplus \mathbb C d$-module $S$ is strongly irreducible if it is irreducible as a (non-graded) 
 $G(\widehat{\mathfrak l})^{\perp}$-module. Verma modules with nonzero central charge and irreducible diagonal modules \cite{BBFK} are examples of such modules.
  
   Though we do not have a classification of all irreducible tensor $\widehat{\mathfrak{l}}$-modules, the following easy statement shows how to construct families of such modules.
  
  \begin{proposition}\label{prop-irr-tensor}
Let $M$ be an irreducible weight  $\widehat{\mathfrak{l}^0}$-module with central charge $a$, $S$  a  strongly irreducible  $\mathbb Z$-graded $G(\widehat{\mathfrak l})^{\perp}\oplus \mathbb C d$-module with the same central charge. Then   $M \otimes S$  is an irreducible weight $\widehat{\mathfrak{l}}$-module with diagonalizable tensor product action of $d$.
 \end{proposition}

We have the following generalization of Theorem \ref{thm-FK}.

\begin{theorem} \label{thm-main}
Let $a\in \mathbb{C}\setminus \{0\}$ and let $V = M \otimes S$ be an irreducible weight tensor $\widehat{\mathfrak{l}}$-module, where
 $M$ is a weight $\widehat{\mathfrak{l}^0}$-module and $S$ is a ${G}(\widehat{\mathfrak l})^{\perp}\oplus \mathbb C d$-module  with diagonalizable action of $d$.   Consider $V$ as a $\widehat{\mathfrak{p}}$-module with trivial action of the radical $\widehat{\mathfrak{u}}$. 
 Then the generalized Imaginary Verma $\widehat{\mathfrak g}$-module $\mathbb{I}_{a, \widehat{\mathfrak{p}}}(V)=M_{a,\widehat{\mathfrak{p}}}(V)$ is irreducible.

\end{theorem}

\begin{proof}
 Consider an arbitrary nonzero element of $v\in M_{a,\widehat{\mathfrak{p}}}(V)$. We can assume that $v$ is a weight element. Write $v$ as the following finite sum
$$v=\sum_{i} d_i (v_i\otimes w_i),$$ where 
 $d_i\in U(\widehat{\overline{\mathfrak{u}}})$, $v_i\in M$, $w_i\in S$, and we assume that $v_i\otimes w_i$ are linearly independent. 
Since $v$ is a weight vector then each $d_i$ is a homogeneous element of $U(\widehat{\overline{\mathfrak{u}}})$.
  Let $R=\mathsf {ht}_{\widehat{\mathfrak{u}}}(v)$. We proceed by the induction on $R$.  Suppose  that $R=1$. 
  As in the proof of Theorem \ref{thm-loop} this is the most important case. Then for all $i$ we have $d_i=x_{-\beta+\alpha_i +r_i\delta}\in \widehat{\mathfrak{g}}_{-\beta+\alpha_i +r_i\delta}$ for some simple root $\beta$ of $\widehat{\mathfrak{g}}$ and some $\alpha_i\in Q(\mathfrak{l})$, where $Q(\mathfrak{l})$ is the root lattice  of the underlined finite dimensional subalgebra $\mathfrak{l}$ of $\widehat{\mathfrak{l}}$. Write $\alpha_i$ as a linear combination of simple roots of $\mathfrak{l}$ with integer coefficients. The $\mathfrak{l}$-height $ht_{\mathfrak{l}}(d_i)$ of $d_i$ is defined as a sum of all coefficients of such linear combination (note that $ht_{\mathfrak{l}}(d_i) \leq 0$ for all $i$).
   Choose $i$ such that $d_i$ has the least $\mathfrak{l}$-height which equals $t$. There can be several $d_j$'s with the same $\mathfrak{l}$-height $t$. Among all of them choose  one with the least $r_i$. Without loss of generality we may assume that $i=1$ satisfies these conditions (not necessarily unique).   We claim that there exists an integer $N$ with sufficiently large $|N|$ and a nonzero  $x_{\beta- \alpha_1 +(N-r_1)\delta} \in \widehat{\mathfrak{g}}_{\beta- \alpha_1 +(N-r_1)\delta}$  (we assume that $\beta- \alpha_1 +(N-r_1)\delta$ is a root and that $|N|>> |r_1|$) such that $x_{\beta- \alpha_1 +(N-r_1)\delta}v\neq 0$. For simplicity write $x_{\beta- \alpha_1 +(N-r_1)\delta}$ by $x_N$.  \ Suppose this is not the case 
  and we have
  $$x_N v=x_N\sum_{i} d_i (v_i\otimes w_i)=\sum_{i} [x_N, d_i] (v_i\otimes w_i)=0$$ for all such $N$.   
  
  Set $y_N^i=[x_N, d_i]$. 
  Note that $y_N^1=[x_N, d_1]\in \widehat{\mathfrak{g}}_{N\delta}$ is nonzero and recall that $\widehat{\mathfrak{l}} = \widehat{\mathfrak{l}}^0 + G(\widehat{\mathfrak l})^{\perp} $ and $[z, G(\widehat{\mathfrak l})^{\perp}] = 0$ for all root elements $z\in \widehat{\mathfrak{l}}^0$.  Write each $y_N^i$ as follows:
  $y_N^i=D_N^i + Z_N^i$, where $D_N^i \in \widehat{\mathfrak{l}}^0$ and $Z_N^i\in G(\widehat{\mathfrak l})^{\perp}$. Clearly, $[D_N^i, Z_N^i]=0$ for all $i$ and $Z_N^1\neq 0$ (it is obvious for any $N$ if $\widehat{\mathfrak{g}}$ is an untwisted affine Lie algebra as $d_{1} \in U(\widehat{\overline{\mathfrak{u}}})$. In the case of twisted affine Lie algebras $N$ can be chosen to guarantee this condition). 
  
 Then 
  $$x_{N}v=v_1\otimes Z_N^1 w_1 + D_N^1 v_1\otimes  w_1 + 
  \sum_{i>1} (v_i\otimes Z_{N}^i w_i+ D_{N}^i v_i\otimes w_i)=0.$$
  Because of the $\mathbb Z$-grading on $S$, the equality above implies
  $$ v_1\otimes Z_N^1 w_1 + \sum_{i>1}v_i\otimes Z_{N}^i w_i =0.$$
 If $v_i$'s are linearly independent then we get a contradiction as either  $Z_N^1 w_1\neq 0$ or  $Z_{-N}^1 w_1\neq 0$. If they are linearly dependent then choose a maximal linearly independent subset
 indexed for simplicity as follows $v_1, v_2, \ldots, v_k$, and
  rewrite the sum above as 
  $$ v_1\otimes (\sum_{j\in Y_1}a_j^1 Z_N^j w_j) + \sum_{i=2}^k v_i\otimes (\sum_{j\in Y_i}a_j^i Z_N^j w_j) =0,$$
 where $Y_1, \ldots, Y_k$ are certain subsets of the set of indices and $a_j^i$ are nonzero scalars. Using now the 
 linear independence of $v_1, \ldots, v_k$ we get  that 
 $$\sum_{j\in Y_i}a_j^i Z_N^j w_j=0$$
 for all $i=1, \ldots, k$.  Consider for example $i=1$. Rewrite the sum above for $i=1$ in the following form:
  $$\sum_{j\in \tilde{Y}_1}a_j^1 Z_N^j \tilde{w}_j=0, $$
  where all $\tilde{w}_j$ are weight elements of $S$ with different weights and $\tilde{Y}_1$ is a subset of $Y_1$. 
  Then we get a contradiction with Lemma \ref{lem-heis}. Hence we can assume that for all $j\in Y_1$, the elements $w_j$  have the same weight in $S$ and $Z_N^j\in \widehat{\mathfrak{g}}_{N\delta}$. Moreover, as $\alpha_r=\alpha_s$ for all $r, s \in Y_1$, then 
    $Z_N^r$ and $Z_N^s$ are the same up to a scalar.
  Suppose that 
  $w_j$, $j\in Y_1$ are linearly independent. Then we get 
  $$0=Z_N \sum_{j\in Y_1} \hat{w}_j=Z_{-N} \sum_{j\in Y_1} \hat{w}_j,$$
  for some nonzero  $Z_{\pm N} \in \widehat{\mathfrak{g}}_{\pm N\delta}$ and some linearly independent elements $\hat{w}_j$ $j\in Y_1$. As the central charge nonzero we have $\sum \limits_{j\in Y_1} \hat{w}_j=0$,  which is a contradiction. 
 Finally, if  $w_j$, $j\in Y_1$ are linearly dependent then $v_1\otimes w_j$ are linearly dependent which contradicts our original assumption. This completes the proof for $R=1$.
 
 Suppose now that $R>1$. Since the proof of the induction step in \cite[Lemma 5.3]{BBFK} does not depend on the module $V$, the same argument applies in our case. Hence there exists $x \in \widehat{\mathfrak{u}}$ such that $xv\neq 0$.
 Since $\mathsf {ht}_{\widehat{\mathfrak{u}}}(xv)=R-1$, the proof is completed by induction. 
 
\end{proof}

 Denote by $\mathcal T_a(\widehat{\mathfrak{l}})$  the category of tensor $\widehat{\mathfrak{l}}$-modules with central charge $a$. The restriction of the imaginary induction functor $\mathbb{I}_{a, \widehat{\mathfrak{p}}}$ 
 on $\mathcal T_a(\widehat{\mathfrak{l}})$ defines the functor  $\mathbb{I}_{a, \widehat{\mathfrak{p}}}^{\mathcal T}$   from the category of tensor $\widehat{\mathfrak{l}}$-modules  with central charge $a$ to the category of weight $\widehat{\mathfrak g}$-modules.
 
As a consequence of Theorem \ref{thm-main} we immediately obtain the following corollary which implies Theorem \ref{thm-2}.

\begin{corollary}\label{cor-functor-irred}
Let $\widehat{\mathfrak p} =\widehat{\mathfrak l}\oplus  \widehat{\mathfrak u}$ be a parabolic subalgebra of 
$\widehat{\mathfrak g}$ of type II, $V\in\mathcal T_a(\widehat{\mathfrak{l}})$ an irreducible $\widehat{\mathfrak{l}}$-module and  $a\in \mathbb{C}\setminus \{0\}$. Then  the functor
 $\mathbb{I}_{a, \widehat{\mathfrak{p}}}^{\mathcal{T}}$  preserves the irreducibility. 
\end{corollary}

\medskip

\subsection{Localization of generalized Imaginary Verma modules}

    Let $\widehat{\mathfrak{g}}$ be an affine Kac-Moody algebra.  For a real positive root $\alpha \in \Delta^{\rm re}_+$  of $\widehat{\mathfrak{g}}$, we define the \emph{twisting functor} 
\begin{align*} 
  T_\alpha=T_{f_\alpha} \colon \mathcal{M}(\widehat{\mathfrak{g}}) \rightarrow \mathcal{M}(\widehat{\mathfrak{g}}) \qquad \text{and} \qquad T_{-\alpha} =T_{e_\alpha} \colon \mathcal{M}(\widehat{\mathfrak{g}}) \rightarrow \mathcal{M}(\widehat{\mathfrak{g}}) 
\end{align*} 
by 
\begin{align*} 
  T_\alpha(M) = U(\widehat{\mathfrak{g}})_{(f_\alpha)}/U(\widehat{\mathfrak{g}}) \otimes_{U(\widehat{\mathfrak{g}})}\   M, 
  \qquad 
  T_{-\alpha}(M) = U(\widehat{\mathfrak{g}})_{(e_\alpha)}/U(\widehat{\mathfrak{g}}) \otimes_{U(\widehat{\mathfrak{g}})}\   M 
\end{align*} 
for $\alpha \in \Delta^{\rm re}_+$ and $M \in \mathcal{M}(\widehat{\mathfrak{g}})$.

Consider a parabolic subalgebra $\widehat{\mathfrak{p}}=\widehat{\mathfrak{l}}\oplus \widehat{\mathfrak{u}}$ of $\widehat{\mathfrak{g}}$ of type II.
Denote by $\Delta^{\rm re}(\widehat{\mathfrak{l}})\subset \Delta^{\rm re}$ the set of real roots of $\widehat{\mathfrak{l}}$.
 The link between the twisting functor $T_\alpha$ for $\alpha \in \Delta^{\rm re}(\widehat{\mathfrak{l}})$ and the Imaginary induction functor $ \mathbb{I}_{a, \widehat{\mathfrak{p}}}$ is given in the following theorem.

\medskip

\begin{theorem}\label{thm:twisting functor intertwining}
Let $\alpha \in \Delta^{\rm re}(\widehat{\mathfrak{l}})\subset \Delta^{\rm re}$. Then there exists a natural isomorphism
\begin{align*}
  \eta_\alpha \colon T_\alpha \circ  \mathbb{I}_{a, \widehat{\mathfrak{p}}} \rightarrow  \mathbb{I}_{a, \widehat{\mathfrak{p}}} \circ \  T_\alpha^{\widehat{\mathfrak{l}}}
\end{align*}
of functors, where $T_\alpha^{\widehat{\mathfrak{l}}} \colon \mathcal{M}(\widehat{\mathfrak{l}}) \rightarrow \mathcal{M}(\widehat{\mathfrak{l}})$ is the twisting functor for $\widehat{\mathfrak{l}}$ assigned to $\alpha$. 

\end{theorem}

\begin{proof} We will follow the proof of \cite[Theorem 4.16]{Futorny-Krizka2021}  as the argument works for any parabolic subalgebra of $\widehat{\mathfrak{g}}$. Let $\alpha \in \Delta^{\rm re}(\widehat{\mathfrak{l}})$ be a positive root and $f_{\alpha}\in \widehat{\mathfrak{g}}_{\alpha}$ a nonzero element. Case of negative roots can be treated similarly. 
By  the Poincar\'e--Birkhoff--Witt theorem we have the following isomorphism
\begin{align*}
  U(\widehat{\mathfrak{g}})_{(f_\alpha)} \simeq U(\widehat{\mathfrak{\overline{u}}})_{(f_\alpha)} \otimes_\mathbb{C}U(\widehat{\mathfrak{p}})
\end{align*}
of $U(\widehat{\mathfrak{\overline{u}}})$-modules. Then for a $\widehat{\mathfrak{g}}$-module $V$ we have
\begin{align*}
  U(\widehat{\mathfrak{g}})_{(f_\alpha)} \otimes_{U(\widehat{\mathfrak{g}})} \mathbb{I}_{a,\widehat{\mathfrak{p}}}(V) \simeq  U(\widehat{\mathfrak{g}})_{(f_\alpha)} \otimes_{U(\widehat{\mathfrak{g}})} U(\widehat{\mathfrak{g}}) \otimes_{U(\widehat{\mathfrak{p}})} V \simeq U(\widehat{\mathfrak{\overline{u}}})_{(f_\alpha)} \otimes_\mathbb{C}V.
\end{align*}

Hence we get an isomorphism
\begin{align*}
  (T_\alpha \circ \mathbb{I}_{a,\widehat{\mathfrak{p}}})(V) \simeq (U(\widehat{\mathfrak{\overline{u}}})_{(f_\alpha)}/U(\widehat{\mathfrak{\overline{u}}})) \otimes_\mathbb{C}V
\end{align*}
of $U(\widehat{\mathfrak{\overline{u}}})$-modules. But we also have an isomorphism of $U(\widehat{\mathfrak{\overline{u}}})$-modules
\begin{align*}
  (\mathbb{I}_{a,\widehat{\mathfrak{p}}} \circ T_\alpha^{\widehat{\mathfrak{l}}})(V) \simeq U(\widehat{\mathfrak{\overline{u}}}) \otimes_\mathbb{C}T_\alpha^{\widehat{\mathfrak{l}}}(V).
\end{align*}

The correspondence
\begin{align*}
  f_\alpha^{-n}u \otimes v \mapsto \sum_{k=0}^\infty (-1)^k \binom{n+k-1}{k} \operatorname{ad}(f_\alpha)^k(u) \otimes f_\alpha^{-n-k}v
\end{align*}
defines the isomorphism of $U(\widehat{\mathfrak{\overline{u}}})$-modules and of $U(\widehat{\mathfrak{g}})$-modules
with the inverse
\begin{align*}
  u \otimes f_\alpha^{-n}v \mapsto \sum_{k=0}^\infty \binom{n+k-1}{k} f_\alpha^{-n-k}\operatorname{ad}(f_\alpha)^k(u) \otimes v
\end{align*}
for all $u \in U(\widehat{\mathfrak{\overline{u}}})$, $v \in V$ and any $n \in \mathbb{N}$.
\end{proof}

\section{Realization of induced modules}

\subsection{Generalized Imaginary Wakimoto modules}

Let $\mathfrak{g}$ be a simple finite dimensional Lie algebra with Cartan decomposition $\mathfrak{g}=\overline{\mathfrak{n}}\oplus \mathfrak{h}\oplus \mathfrak{n}$.
From now on we assume that $\widehat{\mathfrak{g}}$ is untwisted affine Kac-Moody algebra, that is
$$\widehat{\mathfrak{g}}=\mathfrak{g}\otimes \mathbb{C}[t, t^{-1}]\oplus \mathbb{C}c\oplus \mathbb{C} d,$$ 
where $c$ is the central element and $d$ is a degree derivation: $[d, x\otimes t^n]=n(x\otimes t^n)$. It will be convenient to replace the ring of Laurent polynomials $\mathbb{C}[t, t^{-1}]$ by the field $\mathbb{C}((t))$. 

The natural Borel subalgebra $\widehat{\mathfrak{b}}_{\textrm{nat}}$ is defined by $\widehat{\mathfrak{b}}_{\textrm{nat}}  =  H \oplus \widehat{\mathfrak{n}}_{\textrm{nat}}$, where
\begin{eqnarray*}
\widehat{\mathfrak{n}}_{\textrm{nat}} & = & \Big(\mathfrak{n} \otimes_{\mathbb{C}} \mathbb{C}((t)) \Big) \oplus \Big( \mathfrak{h} \otimes_{\mathbb{C}} t\mathbb{C}[[t]] \Big), \\
H & = & \left( \mathfrak{h} \otimes_{\mathbb{C}} \mathbb{C}1\right) \oplus \mathbb{C}c \oplus \mathbb{C}d,\\
\widehat{\overline{\mathfrak{n}}}_{\textrm{nat}} & = & \Big( \overline{\mathfrak{n}} \otimes_{\mathbb{C}} \mathbb{C}((t)) \Big) \oplus \Big( \mathfrak{h} \otimes_{\mathbb{C}} t^{-1}\mathbb{C}[t^{-1}] \Big).
\end{eqnarray*}

Let $\mathfrak{p} = \mathfrak{l} \oplus \mathfrak{u}$ be a parabolic subalgebra of $\mathfrak{g}$, where $\mathfrak{l} $ is the Levi subalgebra and $\mathfrak{u}$ is the radical of $\mathfrak{p}$. Let $\overline{\mathfrak{u}}$
be the opposite radical of $\mathfrak{p}$.

 Then the \emph{natural parabolic subalgebra} of 
  $\widehat{\mathfrak{g}}$ associated with $\mathfrak{p}$ is defined as
  $$\widehat{\mathfrak{p}}_{\textrm{nat}}
  =\widehat{\mathfrak{l}}_{\textrm{nat}} \oplus \widehat{\mathfrak{u}}_{\textrm{nat}},$$  
where 
$\widehat{\mathfrak{l}}_{\textrm{nat}}  =  \Big(\mathfrak{l} \otimes_{\mathbb{C}} \mathbb{C}((t))  \Big) \oplus \mathbb{C} c \oplus \mathbb{C}d $ 
and 
$\widehat{\mathfrak{u}}_{\textrm{nat}}  =   \mathfrak{u} \otimes_{\mathbb{C}} \mathbb{C}((t))$.

The opposite radical of $\widehat{\mathfrak{p}}_{\textrm{nat}}$ is
$\widehat{\overline{\mathfrak{u}}}_{\textrm{nat}}  =   \overline{\mathfrak{u}} \otimes_{\mathbb{C}} \mathbb{C}((t)) $ and we have the  decomposition  
$$\widehat{\mathfrak{g}} = \widehat{\overline{\mathfrak{u}}}_{\textrm{nat}} \oplus \widehat{\mathfrak{l}}_{\textrm{nat}} \oplus \widehat{\mathfrak{u}}_{\textrm{nat}}.$$

  Note that $\widehat{\mathfrak{p}}_{\textrm{nat}}$ is a parabolic subalgebra of type II since $\widehat{\mathfrak{l}}_{\textrm{nat}}$ contains the Heisenberg subalgebra $G$. 
If $\mathfrak{p} = \mathfrak{b}$, that is, $\mathfrak{l} = \mathfrak{h}$ and $\mathfrak{u} = \mathfrak{n}$, we get  $\widehat{\mathfrak{p}}_{\textrm{nat}} = \widehat{\mathfrak{b}}_{\textrm{nat}} \oplus \Big(\mathfrak{h} \otimes_{\mathbb{C}} t^{-1}\mathbb{C}[t^{-1}]\Big)$.
On the other hand, if $\mathfrak{p} = \mathfrak{g}$ we have that $\widehat{\mathfrak{g}}_{\textrm{nat}} = \widehat{\mathfrak{g}}$.

\medskip

Let $\mathcal{K}=\mathbb C((t))$ and let  $\Omega_\mathcal{K}=\mathbb C((t))\,dt$ be the module of Kähler differentials. For a finite-dimensional complex vector space $U$ we define the infinite-dimensional complex vector spaces $\mathcal{K}(U)=U\otimes_\mathbb C \mathcal{K}$ and $\Omega_\mathcal{K}(U^*)=U^*\otimes_\mathbb C \Omega_\mathcal{K}$. 
Denote by $Pol \, \Omega_\mathcal{K}(U^*)$ the algebra of polynomial functions on $\Omega_\mathcal{K}(U^*)$.

Denote by $\dot\Delta(\mathfrak{u})$ the set of  roots $\alpha$ of $\mathfrak{g}$ such that  $\mathfrak{g}_{\alpha}\in \mathfrak{u}$.
Let $\{f_{\alpha} \ | \ \alpha \in \dot\Delta(\mathfrak{u})\}$ be a basis of $\overline{\mathfrak{u}}$ and let $\{x_{\alpha} \ | \ \alpha \in \dot\Delta(\mathfrak{u})\}$ be the linear coordinate functions on  $\overline{\mathfrak{u}}$ with respect to given basis of $\overline{\mathfrak{u}}$, that is, $x_{\alpha}(f_{\beta}) = \delta_{\alpha,\beta}$. Then the set  $\{f_{\alpha,n} = f_{\alpha} \otimes t^{n} \ | \ \alpha \in \dot\Delta(\mathfrak{u}), n \in \mathbb{Z}\}$ forms a topological basis of $\mathcal{K}(\overline{\mathfrak{u}}) = \widehat{\overline{\mathfrak{u}}}_{\textrm{nat}}$, and the set $\{x_{\alpha,-n}=x_{\alpha} \otimes t^{n-1}dt \ | \ \alpha \in \dot\Delta(\mathfrak{u}), n \in \mathbb{Z}\}$  forms a dual topological basis of $ \Omega_{\mathcal{K}}(\overline{\mathfrak{u}}^{\ast}) \simeq (\widehat{\overline{\mathfrak{u}}}_{\textrm{nat}})^{\ast}$.

Consider the infinite rank Weyl algebra $\mathcal{A}_{\mathcal{K}(\overline{\mathfrak{u}})}$  topologically generated by $$\{x_{\alpha,n}, \partial_{x_{\alpha,n}} \ | \ \alpha \in \dot\Delta(\mathfrak{u}), n \in \mathbb{Z}\},$$
such that $a_{\alpha,n} = \partial_{x_{\alpha,n}}$ and $a_{\alpha,n}^{\ast} = x_{\alpha,-n}$, for $\alpha$ $\in$ $\dot\Delta(\mathfrak{u})$, $n \in \mathbb{Z}$.

For $a$ $\in$ $\mathfrak{g}$ we define the formal distribution $a(z)$ $\in$ $\widehat{\mathfrak{g}}[[z^{\pm 1}]]$ by
$$a(z) = \sum_{n \in \mathbb{Z}} a_{n}z^{-n-1}$$
where $a_{n} = a \otimes t^{n}$ for $n$ $\in$ $\mathbb{Z}$. 
 Introduce the formal distributions $a_{\alpha}(z), a_{\alpha}^{\ast}(z)$ $\in$ $\mathcal{A}_{\mathcal{K}(\overline{\mathfrak{u}})}[[z^{\pm 1}]]$ by
\emph{\begin{eqnarray*}
a_{\alpha}(z) = \sum_{n \in \mathbb{Z}} a_{\alpha,n} z^{-n-1} & \emph{and} & a_{\alpha}^{\ast}(z) = \sum_{n \in \mathbb{Z}} a_{\alpha,n}^{\ast}z^{-n}
\end{eqnarray*}}

Inspired by the Feigin-Frenkel homomorphism for vertex algebras the following homomorphism of associative  algebras was established in \cite{FKS19}.

Let $\sigma: \widehat{\mathfrak{p}}_{\textrm{nat}} \rightarrow \mathfrak{gl}(V)$ be a $\widehat{\mathfrak{p}}_{\textrm{nat}}$-module such that $\sigma(c) = k \cdot \textrm{id}_V$ for $k \in \mathbb{C}$.

\begin{theorem}\cite[Theorem 3.6]{FKS19}\label{thm-FKS}
There exists a homomorphism of associative  algebras
\begin{align*}
  \pi_{\sigma(c),\mathfrak{g}} = \pi \colon U(\widehat{\mathfrak{g}}) \rightarrow \smash{{\mathcal{A}}}_{\mathcal{K}(\overline{\mathfrak{u}})} \widehat{\otimes}_\mathbb C\, U(\widehat{\mathfrak{h}}) 
  \end{align*}

  given by $\pi(c) =  \sigma(c)$ and
	\begin{eqnarray*}
	\pi(b(z)) & = & -\sum_{\alpha \in \dot\Delta(\mathfrak{u})} b_{\alpha}(z) \left[ \dfrac{\textrm{ad}(u(z))e^{\textrm{ad}(u(z))}}{e^{\textrm{ad}(u(z))} - \textrm{id}}(e^{-\textrm{ad}(u(z))}(b))_{\overline{\mathfrak{u}}}\right]_{\alpha} + (e^{-\textrm{ad}(u(z))}b(z))_{\mathfrak{p}} + \nonumber \\
	& & - \left(\dfrac{e^{\textrm{ad}(u(z))} - \textrm{id}}{\textrm{ad}(u(z))}\partial_{z}u(z),b \right)\sigma(c)
	\end{eqnarray*}
	for all $b$ $\in$ $\mathfrak{g}$, where
	\begin{eqnarray*}
	u(z) & = &  \sum_{\alpha \in \dot\Delta(\mathfrak{u})} b_{\alpha}^{\ast}(z) f_{\alpha}
	\end{eqnarray*}
and $[X]_{\alpha}$ is the $\alpha$-th coordinate of $X \in \overline{\mathfrak{u}}$ with respect to the  basis $\{f_{\alpha} \ | \ \alpha \in \dot\Delta(\mathfrak{u})\}$ of $\overline{\mathfrak{u}}$.

\end{theorem}

Using Theorem \ref{thm-FKS} one gets explicit realization of generalized Imaginary Verma modules. Namely, we have the following result.

\begin{theorem}\cite[Theorem 3.14]{FKS19}\label{thm-realization}
 Let $V$ be a continuous $\widehat{\mathfrak{p}}_{\emph{\textrm{nat}}}$-module with central charge  $a$ $\in$ $\mathbb{C}$ and with trivial action of $\widehat{\mathfrak{u}}_{\emph{\textrm{nat}}}$. Then the topological vector space $\textrm{Pol} \, \Omega_{\mathcal{K}}(\overline{\mathfrak{u}}^{\ast}) \otimes_{\mathbb{C}} V$ is a continuous $\widehat{\mathfrak{g}}$-module with central charge $a$ and with the module structure determined by Theorem \ref{thm-FKS}. 
 
 \end{theorem}

Note that $\textrm{Pol}\, \Omega_{\mathcal{K}}(\overline{\mathfrak{u}}^{\ast}) = \textrm{Pol} \Big( (\widehat{\overline{\mathfrak{u}}}_{\textrm{nat}})^{\ast} \Big) \simeq S(\widehat{\overline{\mathfrak{u}}}_{\textrm{nat}}) \simeq \mathbb{C}[\partial_{x_{\alpha,n}} \ | \ \alpha \in \dot\Delta(\mathfrak{u}), n \in \mathbb{Z}].$

Set $W_{a, \widehat{\mathfrak{p}}_{\textrm{nat}}}(V)=\textrm{Pol} \, \Omega_{\mathcal{K}}(\overline{\mathfrak{u}}^{\ast}) \otimes_{\mathbb{C}} V$. The $\widehat{\mathfrak{g}}$-module $W_{a, \widehat{\mathfrak{p}}_{\textrm{nat}}}(V)$ will be called \emph{generalized Imaginary Wakimoto module} due to its similarity with 
the generalized Imaginary Verma  modules. In fact we have an isomorphism of  $\widehat{\mathfrak{g}}$-modules:
$$W_{a, \widehat{\mathfrak{p}}_{\textrm{nat}}}(V)\simeq M_{a,\widehat{\mathfrak{p}}_{\textrm{nat}}}(V)$$
for any $\widehat{\mathfrak{l}}_{\textrm{nat}}$-module $V$.

Families of generalized Imaginary Wakimoto modules were studied in \cite{Cox05}, \cite{CF04}, \cite{CF06}   under the name of \emph{Intermediate Wakimoto modules}. 

Applying Theorem \ref{thm-main} we immediately obtain

\begin{corollary}\label{cor-wak}
Let $V$ be an irreducible tensor module over the Levi factor $\widehat{\mathfrak{l}}_{\emph{\textrm{nat}}}$ with nonzero central charge. Then $W_{a, \widehat{\mathfrak{p}}_{\emph{\textrm{nat}}}}(V)$ is  irreducible.

\end{corollary}

\medskip

\subsection{Imaginary Wakimoto functor}

We define the  following \emph{Imaginary Wakimoto functor}
  
   \begin{align*}
  \mathbb{IW}_{a, \widehat{\mathfrak{p}}_{\textrm{nat}}} \colon \mathcal{M}(a, \widehat{\mathfrak{l}}_{\textrm{nat}}) \rightarrow \mathcal{M}(a, \widehat{\mathfrak{g}}),
   \end{align*}
   which sends an $\widehat{\mathfrak{l}}_{\textrm{nat}}$-module $V$ to the $\widehat{\mathfrak{g}}$-module   $W_{a,\widehat{\mathfrak{p}}_{\textrm{nat}}}(V)$.  

\medskip

$$\xymatrix{ & \mathcal{M}(a, \widehat{\mathfrak{l}}_{\textrm{nat}})  \ar@{->}^{\mathbb{I}_{a, \widehat{\mathfrak{p}}_{\textrm{nat}}} \ \ \textrm{Imaginary induction functor}}[d] \ar@/_0.4cm/[d]_{\textrm{Imaginary Wakimoto functor} \ \ \mathbb{IW}_{a, \widehat{\mathfrak{p}}_{\textrm{nat}}} }&  \\ 
& \mathcal{M}(a, \widehat{\mathfrak{g}}) }$$  \\

\medskip

Consider the following particular cases.

\begin{example}
\begin{enumerate}
\item Let $\widehat{\mathfrak{l}}_{\textrm{nat}}=G+H$ and $V\simeq L(a)$ a highest weight $G+H$-module with highest weight $\lambda$, $\lambda(c)=a$.
Then the Imaginary Verma  module $M_{\widehat{\mathfrak{b}}_{\textrm{nat}}}(\lambda)= \mathbb{I}_{a, \widehat{\mathfrak{b}}_{\textrm{nat}}}(\mathbb C)$ is isomorphic to $W_{a,\widehat{\mathfrak{p}}_{\textrm{nat}}}(L(a))=\mathbb{IW}_{a,\widehat{\mathfrak{p}}_{\textrm{nat}}}(L(a))$.
\item Similarly, for any Levi subalgebra $\widehat{\mathfrak{l}}_{\textrm{nat}}$ and a Verma $\widehat{\mathfrak{l}}_{\textrm{nat}}$-module $M_{\widehat{\mathfrak{l}}_{\textrm{nat}}}(\lambda)$ we have the isomorphisms
$$M_{\widehat{\mathfrak{p}}_{\textrm{nat}}}(M_{\widehat{\mathfrak{l}}_{\textrm{nat}}}(\lambda))= \mathbb{I}_{a, \widehat{\mathfrak{p}}_{\textrm{nat}}}(M_{\widehat{\mathfrak{l}}_{\textrm{nat}}}(\lambda))\simeq \mathbb{IW}_{a,\widehat{\mathfrak{p}}_{\textrm{nat}}}(M_{\widehat{\mathfrak{l}}_{\textrm{nat}}}(\lambda)),$$
where $a=\lambda(c)$.

 \item If in the previous item we take the Wakimoto module $\widehat{\mathfrak{l}}_{\textrm{nat}}$-module $W_{\widehat{\mathfrak{l}}_{\textrm{nat}}}(\lambda)$ with highest weight $\lambda$ then
$ \mathbb{IW}_{a,\widehat{\mathfrak{p}}_{\textrm{nat}}}(W_{\widehat{\mathfrak{l}}_{\textrm{nat}}}(\lambda))$
 is an
 Intermediate Wakimoto module  for the parabolic subalgebra $\widehat{\mathfrak{p}}_{\textrm{nat}}$. For a generic $\lambda$  (that is, when $M_{\widehat{\mathfrak{l}}_{\textrm{nat}}}(\lambda)$ is irreducible) we have an isomorphism $M_{\widehat{\mathfrak{l}}_{\textrm{nat}}}(\lambda)\simeq W_{\widehat{\mathfrak{l}}_{\textrm{nat}}}(\lambda)$ and hence
$$ \mathbb{I}_{a, \widehat{\mathfrak{p}}_{\textrm{nat}}}(M_{\widehat{\mathfrak{l}}_{\textrm{nat}}}(\lambda)) \simeq \mathbb{IW}_{a,\widehat{\mathfrak{p}}_{\textrm{nat}}}(W_{\widehat{\mathfrak{l}}_{\textrm{nat}}}(\lambda)).$$
Moreover, if  $\lambda(c)  \neq 0$ then $\mathbb{IW}_{a,\widehat{\mathfrak{p}}_{\textrm{nat}}}(W_{\widehat{\mathfrak{l}}_{\textrm{nat}}}(\lambda))$ is irreducible  by Theorem \ref{thm-main}.

$$\xymatrix{ &  \mathbb{I}_{a, \widehat{\mathfrak{p}}_{\textrm{nat}}}\Big(M_{\widehat{\mathfrak{l}}_{\textrm{nat}}}(\lambda)\Big) & \textrm{Generalized Imaginary Verma module}\\
 M_{\widehat{\mathfrak{l}}_{\textrm{nat}}}(\lambda) \ar@{->}^{\mathbb{I}_{a, \widehat{\mathfrak{p}}_{\textrm{nat}}}}[ru] \ar@{->}_{\mathbb{IW}_{a, \widehat{\mathfrak{p}}_{\textrm{nat}}}}[rd] &  \ar@{--}[u] \ar@{--}[d] & \textrm{$\simeq$ generically (for a generic $\lambda$)} \\
 & \mathbb{IW}_{a, \widehat{\mathfrak{p}}_{\textrm{nat}}}\Big(W_{\widehat{\mathfrak{l}}_{\textrm{nat}}}(\lambda)\Big) & \textrm{Generalized Imaginary Wakimoto module}}$$

\end{enumerate}
\end{example}

\medskip

 We also get from Theorem \ref{thm-main}
  
  \begin{corollary}\label{cor-ind}
 Let $\widehat{\mathfrak{l}}_{\rm nat} = \widehat{\mathfrak{l}}^0_{\rm nat} + G(\widehat{\mathfrak{l}}_{\rm nat})^{\perp}$, where $[x, G(\widehat{\mathfrak{l}}_{\rm nat})^{\perp}] = 0$, for any root element  $x\in  \widehat{\mathfrak{l}}^{0}_{\rm nat}$. If $N$ is an irreducible $\widehat{\mathfrak{l}}^0_{\rm nat}$-module and $S$ is a strongly irreducible $ G(\widehat{\mathfrak l}_{\rm nat})^{\perp}\oplus \mathbb C d$-module, both with the same nonzero central charge $a$, then
  $$\mathbb{I}_{a, \widehat{\mathfrak{p}}_{\emph{\textrm{nat}}}} \Big( N \otimes S\Big) \simeq \mathbb{IW}_{a, \widehat{\mathfrak{p}}_{\emph{\textrm{nat}}}} \Big( N \otimes S\Big) $$ is irreducible. 
  
  \end{corollary}
  
  \medskip
  
For any root  $\alpha \in \dot{\Delta}_{+}(\mathfrak{l}) \subset \Delta^{\rm re}_{+}(\widehat{\mathfrak{l}}_{\rm nat}) \subset \Delta^{\rm re}_{+}$ \ consider the twisting functor $ T_\alpha$. Then  Theorem \ref{thm:twisting functor intertwining}   and \cite[Theorem 4.16]{Futorny-Krizka2021} immediately imply

\begin{corollary}\label{cor-twisting-wak} 
$
  T_\alpha \circ \mathbb{IW}_{a, \widehat{\mathfrak{p}}_{\emph{\textrm{nat}}}} = \mathbb{IW}_{a, \widehat{\mathfrak{p}}_{\emph{\textrm{nat}}}} \circ\, T_\alpha^{\widehat{\mathfrak{l}}_{\emph{\textrm{nat}}}}.
$
In particular, we have
\begin{align*}
  T_\alpha\Big(\mathbb{IW}_{a, \widehat{\mathfrak{p}}_{\emph{\textrm{nat}}}}(M_{\widehat{\mathfrak{l}}_{\emph{\textrm{nat}}}}(\lambda))\Big) \simeq \mathbb{IW}_{a, \widehat{\mathfrak{p}}_{\emph{\textrm{nat}}}} \Big(T_{\alpha}^{\widehat{\mathfrak{l}}_{\emph{\textrm{nat}}}}({M}_{\widehat{\mathfrak{p}}}(M_{\mathfrak{l}}(\lambda))) \Big)\simeq   \mathbb{IW}_{a, \widehat{\mathfrak{p}}_{\emph{\textrm{nat}}}} \Big({M}_{\widehat{\mathfrak{p}}}(T_{\alpha}^\mathfrak{l}(M_{\mathfrak{l}}(\lambda)))\Big)
\end{align*}

and

\begin{eqnarray*}
  T_\alpha\Big(\mathbb{IW}_{a, \widehat{\mathfrak{p}}_{\emph{\textrm{nat}}}}(N \otimes S)\Big) & \simeq & \mathbb{IW}_{a,\widehat{\mathfrak{p}}_{\emph{\textrm{nat}}}} \Big(T_{\alpha}^{\widehat{\mathfrak{l}}_{\emph{\textrm{nat}}}}(N \otimes S) \Big) \simeq \\
  & \simeq &\mathbb{IW}_{a,\widehat{\mathfrak{p}}_{\emph{\textrm{nat}}}}\Big(T_{\alpha}^{\widehat{\mathfrak{l}}_{\emph{\textrm{nat}}}}(N) \otimes S\Big),
\end{eqnarray*}
where  $M_{\widehat{\mathfrak{l}}_{\emph{\textrm{nat}}}}(\lambda)$ is the Verma $\widehat{\mathfrak{l}}_{\emph{\textrm{nat}}}$-module with highest weight $\lambda$, $M_{\mathfrak{l}}(\lambda)$ is the Verma $\mathfrak{l}$-module with highest weight $\lambda$, $\widehat{\mathfrak{p}}$ is the standard parabolic subalgebra of $\widehat{\mathfrak{l}}_{\emph{\textrm{nat}}}$ with finite-dimensional  Levi factor ${\mathfrak{l}}+H$, ${M}_{\widehat{\mathfrak{p}}}$ is the induction functor $Ind_{\widehat{\mathfrak{p}}}^{\widehat{\mathfrak{l}}_{\emph{\textrm{nat}}}}$ and the modules $N$ and $S$ are as in Corollary \ref{cor-ind}.
\end{corollary}

Note that generically the modules in Corollary \ref{cor-twisting-wak} are irreducible if $a\neq 0$. Hence
Corollaries \ref{cor-ind} and \ref{cor-twisting-wak} provide technique to construct new irreducible modules together with explicit realizations by differential operators (see Theorem \ref{thm-realization}).


 $$ \xymatrix{M_{\widehat{\mathfrak{l}}_{\textrm{nat}}}(\lambda) \ar@{->}_{T_{\alpha}^{\widehat{\mathfrak{l}}_{\textrm{nat}}}}[d] & \simeq & \relax\underbrace{M_{\widehat{\mathfrak{p}} }\Big(M_{\mathfrak{l}}(\lambda)\Big)} \ar@{->}^{T_{\alpha}^{\widehat{\mathfrak{l}}_{\textrm{nat}}}}[d] \\
 \relax\underbrace{T_{\alpha}^{\widehat{\mathfrak{l}}_{\textrm{nat}}}\Big(M_{\widehat{\mathfrak{l}}_{\textrm{nat}}}(\lambda)\Big)} \ar@{->}_{\mathbb{IW}_{a, \widehat{\mathfrak{p}}_{\textrm{nat}}} }[d] & \simeq & M_{\widehat{\mathfrak{p}} }\Big(T_{\alpha}^{\mathfrak{l}} (M_{\mathfrak{l}}(\lambda))\Big) \\
\mathbb{IW}_{a, \widehat{\mathfrak{p}}_{\textrm{nat}}}\Big(T_{\alpha}^{\widehat{\mathfrak{l}}_{\textrm{nat}} }(M_{\widehat{\mathfrak{l}}_{\textrm{nat}}}(\lambda))\Big) & \simeq & \mathbb{IW}_{a, \widehat{\mathfrak{p}}_{\textrm{nat}}} 
 \Big(M_{\widehat{\mathfrak{p}} }\Big(T_{\alpha}^{\mathfrak{l}} (M_{\mathfrak{l}}(\lambda))\Big)\Big)\\} $$


\section*{Acknowledgments}

We express our gratitude to Vyacheslav Futorny, who proposed the problem and with whom we had numerous helpful discussions. 
F. J. S. S. is supported by the FAPEAM grant (006/2018).


\providecommand{\bysame}{\leavevmode\hbox to3em{\hrulefill}\thinspace}
\providecommand{\MR}{\relax\ifhmode\unskip\space\fi MR }
\providecommand{\MRhref}[2]{%
  \href{http://www.ams.org/mathscinet-getitem?mr=#1}{#2}
}
\providecommand{\href}[2]{#2}

\end{document}